\documentclass[a4paper,twoside,11pt]{article}
\usepackage{a4wide}
\usepackage[utf8]{inputenc}
\usepackage[english]{babel}
\usepackage{color} 
\usepackage{amsmath,amssymb,amsthm,amsfonts}
\usepackage{bbm}
\usepackage{graphicx}
\usepackage{hyperref}
\usepackage[d]{esvect}
\theoremstyle{plain}
\newtheorem{theorem}{Theorem}[section]
\newtheorem*{theorem*}{Theorem}
\newtheorem{lemma}[theorem]{Lemma}
\newtheorem{corollary}[theorem]{Corollary}
\newtheorem{proposition}[theorem]{Proposition}

\theoremstyle{definition}

\theoremstyle{remark}
\newtheorem{remark}[theorem]{Remark}

\numberwithin{equation}{section}

\def\N{\ensuremath{\mathbb{N}}}

\def\R{\ensuremath{\mathbb{R}}}
\def\Z{\ensuremath{\mathbb{Z}}}

\def\ep{\varepsilon}

\def\E{\ensuremath{\mathbf{E}}}

\def\P{\ensuremath{\mathbf{P}}}

\def\F{\ensuremath{\mathcal{F}}}

\def\Ind{\ensuremath{\mathbbm{1}}}

\def\to{\rightarrow}

\def\tand{\ensuremath{\text{ and }}}

\def\tas{\ensuremath{\text{ as }}}

\newcommand{\dd}{\mathrm{d}}

\def\Ai{\mathrm{Ai}}

\author{\textsc{Pascal Maillard}\thanks{
Département de Mathématiques, Université Paris-Sud, 91405 Orsay Cedex, France. 
Partially supported by a grant from the Israel Science Foundation}
\and \textsc{Ofer Zeitouni} \thanks{Department of Mathematics, The Weizmann Institute of Science, POB 26, Rehovot 76100, Israel
and Courant Institute, New-York University. Partially supported by a grant from the Israel Science Foundation and the Henri Taubman professorial 
chair at the Weizmann Institute}}
\title{Slowdown in branching Brownian motion
with inhomogeneous variance}
\usepackage[normalem]{ulem}
\date{February 2, 2014. Revised March 6, 2015}
\begin{document}

\maketitle

\begin{abstract} We consider the distribution of the maximum $M_T$
of branching Brownian motion with time-inhomogeneous 
variance of the form $\sigma^2(t/T)$,
  where $\sigma(\cdot)$ is a strictly decreasing
function. This corresponds to the study of the time-inhomogeneous
Fisher--Kolmogorov-Petrovskii-Piskunov (F-KPP)
equation $
F_t(x,t)=\sigma^2(1-t/T)F_{xx}
(x,t)/2+
g(F(x,t))$, for appropriate nonlinearities $g(\cdot)$.
Fang and Zeitouni (2012) showed that 
$M_T-v_\sigma T$ is negative of 
order $T^{1/3}$, where $v_\sigma=\int_0^1 \sigma(s)ds$. In this paper, 
we show the existence of a function $m'_T$, such that $M_T-m'_T$ converges in law, as $T\to\infty$. Furthermore,
$m'_T=v_\sigma T - w_\sigma T^{1/3} - \sigma(1)\log T + O(1)$ 
with $w_\sigma = 
2^{-1/3}\alpha_1 \int_0^1 \sigma(s)^{1/3} |\sigma'(s)|^{2/3}\,\dd s$. Here,
$-\alpha_1=-2.33811...$ is the largest zero of the Airy function $\Ai$. The proof uses a mixture of probabilistic and analytic arguments.
\end{abstract}

\renewcommand{\abstractname}{R\'esum\'e}
\begin{abstract}
 Nous \'etudions la loi du maximum $M_T$ d'un mouvement brownien branchant avec une variance inhomog\`ene en temps de la form $\sigma^2(t/T)$, o\`u $\sigma(\cdot)$ est une fonction strictement d\'ecroissante. Ceci correspond \`a \'etudier l'\'equation Fisher--Kolmogorov-Petrovskii-Piskunov (F--KPP) inhomog\`ene en temps, $F_t(x,t)=\sigma^2(1-t/T)F_{xx}
(x,t)/2+
g(F(x,t))$, pour des nonlin\'earit\'es $g(\cdot)$ appropri\'ees. Fang et Zeitouni (2012) ont montr\'e que $M_T-v_\sigma T$ est negatif de l'ordre $T^{1/3}$, o\`u $v_\sigma=\int_0^1 \sigma(s)ds$. Dans cet article, nous montrons l'existence d'une fonction $m'_T$ telle que $M_T-m'_T$ converge en loi quand $T\to\infty$. De plus, $m'_T=v_\sigma T - w_\sigma T^{1/3} - \sigma(1)\log T + O(1)$ avec $w_\sigma = 
2^{-1/3}\alpha_1 \int_0^1 \sigma(s)^{1/3} |\sigma'(s)|^{2/3}\,\dd s$. Ici, $-\alpha_1=-2.33811...$ est la plus grande racine de la fonction d'Airy $\Ai$. La d\'emonstration repose sur un m\'elange d'arguments probabilistes et analytiques.
\end{abstract}

\section{Introduction}
The classical branching Brownian motion (BBM)
model in $\mathbb{R}$
can be described probabilistically as follows. Fix a law 
$\mu$ of finite variance on $[2,\infty)\cap \Z$. At time $t=0$,
one particle exists and is located at the origin. This particle
starts performing standard Brownian motion on the real line, up to an exponentially distributed random
time, with parameter $\beta_0 = (2(\E_\mu[L]-1))^{-1}$ (that is,
branching occurs at rate $\beta_0$).
At that time, the particle instantaneously splits into a random number 
$L\geq 2$ of
independent
particles, and those start afresh performing Brownian motion until their
(independent) exponential clocks ring. There is an extensive literature on this model and its discrete analog, the branching random walk, in particular concerning the position of the right-most particle (see e.g.\ \cite{M75,Bramson78,Bramson83,DS,Roberts,Aidekon}). In order to state the main result, introduce the F-KPP travelling wave equation
\begin{equation}
\label{eq:fkpp}
\phi:\R\to(0,1)\ \text{increasing},\quad \tfrac 1 2 \phi'' +
\phi'+\beta_{0}(\E_\mu[\phi^L]-\phi) = 0,
\quad\phi(-\infty)=0,\ \phi(+\infty)=1.
\end{equation}
One has the following theorem:
\begin{theorem*}[Bramson \cite{Bramson83}]
Let $M_t$ denote the position of the right-most particle at time $t$ in branching Brownian motion as defined above. Then there exists a solution $\phi$ to \eqref{eq:fkpp}, such that for all $x\in\R$,
\[
\P(M_t \le t-\tfrac 3 2 \log t + x) \to \phi(x),\quad\tas t\to\infty.
\]
\end{theorem*}

We discuss in this paper a variant of the BBM
model, first introduced in \cite{DS}, where the motion of the
particle(s) is controlled by a time-inhomogeneous variance. More precisely, 
let $\sigma\in C^2([0,1])$ be a strictly decreasing function with
$\sigma(1)>0$ and $\inf_{t\in [0,1]} |\sigma'(t)|>0$. 
We assume that the variance of the Brownian motions at time 
$t\in [0,T]$ is given by $\sigma^2(t/T)$. 
%This model has first been considered in \cite{DS}. 

Let ${N(t)}, t\in [0,T]$
denote the collection of particles
alive at time $t$
%, set $n(t)=| N(t)|$,
and for any particle $v\in N(t)$, let $X_v(s), s\in [0,1]$
denote the  trajectory performed by the particle and its ancestors.
%$n(t)$ is a continuous time branching process, and
%it is straightforward to verify that $n(t)e^{-t/2}$ is a martingale, %which
%converges almost surely to a positive, finite
%random variable $n_\infty$. In particular, due to our choice of parametrization
%$\beta_0$,
%$t^{-1}\log n(t)$ converges
%almost surely to $1/2$.
Then $M_t
=\max_{u\in  N(t)} X_u(t)$ 
denotes the location of the rightmost
particle at time $t$. The cumulative distribution function of $M_T$ is $F(\cdot,T)$, where $F(x,t)$ is the solution of the 
time-inhomogeneous Fisher--Kolmogorov-Petrovskii-Piskunov (F--KPP)
equation
\begin{eqnarray}
  \label{eq-FKPP}
  \frac{\partial F}{\partial t}(x,t)&=&\frac{\sigma^2(1-t/T)}
2 \frac{\partial^2 F}{\partial^2 x}
(x,t)+
\beta_{0}(\E_\mu[F(x,t)^L]-F(x,t))
\,, t\in [0,T], x\in \mathbb{R}\nonumber\\
F(x,0)&=&{\bf 1}_{x\geq 0}\,.
\end{eqnarray}
See \cite{M75} for this probabilistic interpretation of the F--KPP equation in
the time homogeneous case.

In \cite{FZ12}, 
the authors prove the following.
\begin{theorem*}[Fang, Zeitouni \cite{FZ12}]
  \label{theo-FZ}
  There exist constants $C,C'>0$ so that
  \begin{equation}
    \label{eq-FZ1}
    -C\leq \liminf_{T\to\infty} \frac{M_T-v_\sigma T}{T^{1/3}}
    \leq \limsup_{T\to\infty} \frac{M_T-v_\sigma T}{T^{1/3}}
    \leq -C'<0\,,
  \end{equation}
  where $v_\sigma = \int_0^1\sigma(s) ds$.
\end{theorem*}
(The derivation in \cite{FZ12} is for the case that $P(L=2)=1$, but applies
with no changes to the current setup. The linear in $T$ asymptotics, i.e.
the speed $v_\sigma$, can be read off with some effort from the results 
in \cite{DS} and \cite{BK}.)

Our goal in this paper is to significantly refine Theorem 
\ref{theo-FZ}.
To state our results, introduce the functions
%\section{Model}
%BBM starting with one particle at the origin, reproduction according to r.v. $L\ge2$, $\E[L^2]<\infty$, branching rate $\beta_0 = (2(\E[L]-1))^{-1}$, time interval $[0,T]$, variance $\sigma^2(t/T)$, $\sigma^2\in C^2[0,1],$ strictly decreasing, \(\sigma(1)>0 \).
%Denote by \(M_{T} \) the position of the maximal particle at time \(T,\) i.e. \(M_{T }=\max_{u\in N(T)}X_u(T)\), where \(X_u(t)\) denotes the position of the individual \(u\) at time \(t\) and \(N(t)\) is the set of individuals alive at time \(t.\)
%
%Define the functions 
$v,w:[0,1]\to\R_+$ by 
\begin{equation}
  \label{eq-v}
  v(t) = \int_0^t  \sigma(s)\,\dd s\,,
\end{equation}
and 
\begin{equation}
  \label{eq-w}
  w(t) = 2^{-1/3}\alpha_1 \int_0^t \sigma(s)^{1/3} 
|\sigma'(s)|^{2/3}\,\dd s\,,
\end{equation}
where $-\alpha_1=-2.33811...$ is the largest zero of the 
\textit{Airy function of the first kind}
\begin{equation}
  \label{eq-Airy}
  \Ai(x)=\frac1\pi \int_0^\infty \cos\left(\frac{t^3}{3}+xt\right)
dt\,,
\end{equation}
see
%\cite[Section 5.17]{Lebedev} or 
\cite[Section 10.4]{AbrSteg} for definitions;
note that $\Ai$ satisfies the
Airy differential equation $\Ai''(x)-x\Ai(x)=0$. Note also that $v_\sigma=v(1)$.
Set
\[
m_T = v(1) T - w(1)T^{1/3} - \sigma(1)\log T.
\]
Our main result is the following.
\begin{theorem}
\label{th:1}
The family of random variables  
\((M_T- m_T)_{T\geq0 }\) is tight. 
Further, there exists a solution \(\phi(x)\) to \eqref{eq:fkpp} and a function $m'_T$ with $C_\sigma = \limsup_{T\ge0}|m'_T - m_T| < \infty$, such that for all $x\in\R$,
\[
\lim_{T\to\infty}\P(M_T\le m'_T+x)=\phi(x/\sigma(0)).
\]
Furthermore, for a fixed travelling wave $\phi$, the constant $C_\sigma$ above is uniformly bounded for $$\sigma\in \{\sigma\in C^2([0,1]): \sigma(0)+1/\sigma(1)<c_0, \sup_{t\in [0,1]}|\sigma''(t)|<c_0, \inf_{t\in [0,1]}
|\sigma'(t)|>1/c_0\}=:\Xi_{c_0}\,.$$
%for all $\sigma$ in a family $(\sigma_i)_{i\in I}$, 
%such that $1/\sigma_i(1)$ and 
%$\sup_{t\in[0,1]}\sigma_i'(t)$ are uniformly bounded.
\end{theorem}
Parallel to our work, and an inspiration to it, was the study \cite{NRR}, by 
PDE techniques, of a class of time-inhomogeneous F--KPP equations that includes
\eqref{eq-FKPP}. Compared with
\cite{NRR}, we deal with a slightly restricted class of equations, but are able to obtain finer (up to order $1$) asymptotics and convergence to a travelling wave. We hope that our techniques 
can be pushed to yield convergence in distribution of
the family $(M_T-m_T)_{T\ge0}$ (instead of $(M_T-m'_T)_{T\ge0}$), in parallel with the recent results in 
\cite{BDZ}, but this requires significant changes in the approach of \cite{BDZ}
(mainly, because unlike in the time-homogeneous case, extremal particles at time
$T$ will, with positive probability, be extremal at some random intermediate
time between $\epsilon T$ and $(1-\epsilon)T$). We therefore leave the
adaptation for possible future work.

We remark that Mallein \cite{Mallein} has recently published results similar to ours which are less precise but hold for a rather general class of (not necessarily Gaussian) time-inhomogeneous branching random walks.

The core of the proof of Theorem \ref{th:1}  is based on a constrained first and second moment analysis of the number of 
particles that reach a  target value but remain below a barrier for the duration
of their lifetime. 
%Recall that
%with 
%$n(t)$ denoting the number of particles alive at time $t$, we have that
%Let ${N(t)}, t\in [0,T]$
%denote the collection of particles
%alive at time $t$, set $n(t)=| N(t)|$,
%and for any particle $v\in N(t)$, let $X_v(s), s\in [0,1]$
%denote the  trajectory performed by the particle and its ancestors.
%$n(t)$ is a continuous time branching process, and
%it is straightforward to verify that 
%$n(t)e^{-t/2}$ is a martingale, 
%which
%converges almost surely to a positive, finite
%random variable $n_\infty>0$. 
%In particular, ]
%due to our choice of parametrization
%$\beta_0$,
%$t^{-1}\log n(t)$ converges
%almost surely to $1/2$.
%Thus, first moment computations essentially reduce to evaluating
%probabilities of individual paths.
%
Due to the time-inhomogeneity of $\sigma(\cdot)$, 
the choice of barrier is not  straight-forward, and in particular
it is 
not a 
straight line; ``rectifying'' it introduces a killing potential. The
analysis of the survival of Brownian motion in this potential
eventually leads to a time-inhomogeneous Airy-type differential equation which we study by analytic means, exploiting the anti-symmetry of the differential operator. (As pointed out to us by Dima Ioffe, a 
similar
phenomenon with related $T^{1/3}$ scaling
was already observed in \cite{Gro,spohnferrari}.) These methods together lead to estimates of the right tail of $M_T$ which are sharp up to a multiplicative factor (Proposition~\ref{prop:tail}). By a bootstrapping procedure that may be of independent interest, these estimates are then turned into  convergence in law by using a convergence result for the derivative Gibbs measure of (time-homogeneous) branching Brownian motion.

The structure of the paper is as follows. In the next section, we introduce 
a barrier $\gamma_T(\cdot)$, and show that with high probability, no particle
crosses (a shifted version of) the barrier, see Lemma \ref{lem:curve}. Using 
the barrier, we then control the distribution of extremal particles at all 
times large enough
(Lemma \ref{lem:density}). In these lemmas, results
concerning time-inhomogeneous Airy-type PDE's are needed, and the proof of 
those is given in Appendix A (Section \ref{app}). 
Section \ref{sec-4} combines the results of  Section \ref{sec-3}
(taken at time 
$T-T^{2/3}$) together with an analysis of the last segment of time of
length $T^{2/3}$, and provides the first-and-second moment
results needed to obtain lower and upper bound on the right tail of $M_T$.
The proof of Theorem \ref{th:1} is then completed in Section \ref{sec-5}, using a result about the convergence of the derivative Gibbs measure of (time-homogeneous) branching Brownian motion, which is given in Appendix B (Section~\ref{sec:derivative_Gibbs}).

\paragraph{Notation} In the rest of this article (except in the appendix), the symbols $C$,$C'$,$C_1$,$C_2$ etc.\ stand for positive constants, possibly depending on $c_0$ (see Theorem~\ref{th:1}), whose values may change from line to line. The phrase ``$X$ holds for large $T$'' means that there exists $T_0$, possibly depending on $c_0$, such that $X$ holds for $T\ge T_0$ for all $\sigma\in \Xi_{c_0}$. We further use the Landau symbols $O(\cdot)$ and $o(\cdot)$, which are always to be interpreted with respect to $T\to\infty$, and which may depend on $c_0$ as well. Finally, the symbols $\P$ and $\E$ (possibly with sub-/superscripts) always stand for the law of a branching Markov process (branching Brownian motion with time-varying or constant variance and with or without absorption of particles) and the  expectation with respect to this law. On this other hand, the symbols $P$ and $E$ are used for probability and expectation with respect to a single particle (i.e.\ a Markov process, usually a Brownian motion with time-varying or constant variance or a three-dimensional Bessel process). The location of the initial particle is denoted by a subscript, e.g.\ $\P_x$, without a subscript the initial particle is implicitly located at the origin.

\paragraph{Acknowledgements}
We thank Lenya Ryzhik for very stimulating conversations concerning 
the PDE approach to time-inhomogeneous BBMs, and for making \cite{NRR} available
to us before we completed work on this paper. We also thank Bastien Mallein
for describing to us his progress on analogous questions for 
branching random walks.

\section{Crossing estimates}
%The process until the time \texorpdfstring{$T-T^{2/3}$}{T-T\^{}(2/3)}}
\label{sec-3}

Fix $T$. 
Define the curve $\gamma_T:[0,T]\to\R$ by 
\[\gamma_T(t) = Tv(t/T) - T^{1/3}w(t/T).\] 
Introduce the constant
\begin{equation}
\label{eq:s_0}
\kappa := 8/\sigma^{2}(1).
\end{equation}

In this section we prove two lemmas. The first lemma bounds, for any 
fixed $K\geq 1$, the probability that there exists a particle that
reaches the curve 
$\gamma_T(t)+K$.
The second lemma
estimates the expected number of particles that have stayed below the curve
up to time $t$, and reach a given terminal value at time $t$.

\begin{lemma}
\label{lem:curve}
There exists a constant $C=C(c_0)$, 
such that for large \(T,\) for any $\sigma\in \Xi_{c_0}$ and
every $K\in[1,T^{1/3}]$,
\[\P(\exists t\in[0,T]: \max_{u\in N(t)}X_u(t) \ge \gamma_T(t)+K) 
\le CKe^{-K/\sigma(0)}.\]
\end{lemma}
\begin{proof}
The proof goes by a first moment estimate of the number 
of particles hitting the curve \(\gamma_T+K\).
For an interval $I\subset[0,T]$, let \(R_I\) be the number of particles 
hitting the curve \(\gamma_T+K\) for the first time during the interval $I$. 
Let $B_t$ be a Brownian motion with variance $\sigma^2(t/T)$ 
started from the point $x$ under $P_x$ (see the remarks on notation in the introduction).
For a path $(X_t)_{t\ge0}$, 
define $H_{0}(X) = \inf \{t\ge 0: X_t =0 \}$.
By the first moment formula\footnote{Note that due to our choice of the branching rate $\beta_0$, the expected number of particles in the system at the time $t$ is $\E[N(t)]=e^{t/2}$, which is the reason for the exponential term arising in the formula.} for branching Markov processes \cite[Theorem~4.1]{INW3} (also known as ``Many-to-one lemma'') we then have (taking $x=K$)
$$\E[R_{I}] =E_0\Big[e^{H_0(\gamma_T+K - B)/2} 
\Ind_{H_0(\gamma_T+K - B)\in I}\Big]=E_K\Big[e^{H_0(B+\gamma_T)/2} 
\Ind_{H_0(B+\gamma_T)\in I}\Big],$$
where the second equality follows from the fact that the law of $K-B_t$ under $P_0$ is equal to the law of $B_t$ under $P_K$ by symmetry.
Applying Girsanov's theorem we get that
\begin{align*}
\E[R_{I}] 
&=E_K\left[\exp\Big(\int_0^{H_0(B)}\frac{\gamma_T'(t)}{\sigma^2(t/T)}\,\dd 
B_t + \frac {H_{0}(B)} 2 - \int_0^{H_0(B)}\frac{(\gamma _T'(t))^2}{2\sigma^2(t/T)}
\,\dd t \Big)\Ind_{H_{0}(B)\in I}\right]\\
&= e^{-K\gamma_T'(0)/\sigma^2(0)+o(1)}E_K\left[\exp\left(  
\frac 1 T\int_0^{H_{0}(B)}
\left(-q_T(t/T)B_t\,+T^{1/3}\frac{w'(t/T)}{\sigma(t/T)}\right)
\,\dd t  \right)\Ind_{H_{0}(B)\in I}\right],
\end{align*}
where the last equation follows by integration by parts and the function $q_T:[0,1]\to\R$ is defined by
\[q_T(t) = \frac {|\sigma'(t)|}{\sigma^2(t)} + T^{-2/3}(w'/\sigma^2)'(t). \] 
For large $T$, this yields
by \eqref{eq-w} and the assumptions on $\sigma$ and $K$,
\begin{multline}
\label{eq:145}
\E[R_{I}] = e^{-K/\sigma(0)+o(1)} E_K\Big[\exp\Big(  \frac 1
 T\int_0^{H_{0}(B)}\Big\{-q_T(t/T)B_t\\
+\alpha_1q_{T}(t)^{2/3}\Big( \tfrac 1 2 
\sigma^2(t/T) \Big)^{1/3}T^{1/3}\Big\}\,\dd t  \Big)\Ind_{H_{0}(B)\in I}\Big]
\end{multline}
Set
\begin{equation}
  \label{eq-J}
  J(t) = \int_0^t \tfrac 1 2 \sigma(s)^2\,\dd s\,.
\end{equation}
%Let $J$ be the time-change from the end of the previous section and 
%Define 
%$s_0 := J^{-1}(4T^{-1/3}J(1))T$, so that $CT^{2/3}\le s_0 \le C'T^{2/3}$ 
%$s_0 := 4\sigma^2(0)\sigma^{-2}(1)T^{2/3}$.
%, so that $CT^{2/3}\le s_0 \le C'T^{2/3}$ 
%for large $T$. 
Recall the constant $\kappa$ from \eqref{eq:s_0}.
We will bound separately $\E[R_{[0,\kappa T^{2/3}]}]$ and $\E[R_{[\kappa T^{2/3},T]}]$. 
For the first term, \eqref{eq:145} immediately gives 
\begin{equation}
  \label{eq-ERt0}
  \E[R_{[0,\kappa T^{2/3}]}] \le Ce^{- K/\sigma(0)} \,, 
\end{equation}
because under $P_K$,
\(B_t\) is positive until the time \(H_0(B)\) and the factor in front of $T^{1/3}$ in the integral in  \eqref{eq:145} is bounded by a constant $C$.
In order to bound $\E[R_{\kappa T^{2/3},T]}]$, we note that the expectation on the 
right side of \eqref{eq:145}  equals
\begin{equation}
\label{eq:209}
\int_I \tfrac 1 2 {\sigma^2(t/T)}\frac{\dd G(K,y;t)}{\dd y}\Big|_{y=0}\,\dd t,
\end{equation}
where $G(x,y;t)$ is the fundamental solution to the 
PDE \eqref{eq:pde}, with $Q(t) = q_T(J(t/T))$ (see \cite[Sections~5.2.1 and 5.2.8]{Gardiner} for an elementary, but somewhat non-rigorous proof of this fact, and \cite[Sections~I.XI.7 and 2.IX.13]{Doob} for the formal definition of parabolic measure and its relation to hitting time distributions for Brownian motion).
Now, by \eqref{eq:ihp2} of Corollary~\ref{cor:G_estimate_1},
%Proposition~\ref{prop:fund_estimate} and 
%\eqref{eq:green}, there exists a positive function $Q_*(t)\le Q(t)$, such that for every $t\in [s_0,T]$,
\[
\frac{\dd G(K,y;t)}{\dd y}\Big|_{y=0} \le 
%T^{-1} C\psi_1\big(Q_*(t)K\big) \sum_{n=1}^\infty e^{-C'n^{2/3}}\psi_n'(0) i
%\le i
CT^{-1} K\,,
\]
for any $t\in [\kappa T^{2/3},T]$.
%where the last inequality follows from Lemma~\ref{lem-bullets}.
Together with  \eqref{eq:145} and \eqref{eq:209}, this yields for large $T$, 
\begin{equation}
\label{eq:884}
\E[R_{[\kappa T^{2/3},T]}]\le CKe^{-K/\sigma(0)}.
\end{equation}
The lemma now follows from \eqref{eq-ERt0}  and \eqref{eq:884} and Markov's inequality.
\end{proof}

We next control the expected number of particles that stay below the curve 
$\gamma_T(\cdot)+K$ up to time $t\leq T$ and reach a 
prescribed value at time $t$. In what follows, for measures $\mu,\nu$ we use the notation
$\mu(\cdot\in dy)\leq \nu(\cdot\in dy)$, $y\geq 0$,
to mean that for any interval $I\subset \R_+$,
$\mu(\cdot\in I)\leq \nu(\cdot\in I)$.
\begin{lemma}
\label{lem:density}
For large $T$, we have for all $t\in[0,T]$, $K\in[0,T^{1/3}]$ and $y>0$,
% Then there exist 
% constants \(C,C'>0\) (depending on $c_0$ only),
% such that for large $T$ and for all $K\in[0,T^{1/3}]$ and $y>0$,
\begin{multline*}
%\sqrt{q_T(0)} Ke^{-K/\sigma(0)}T^{-2/3}\Big[\psi_1^{q^{*}(t/T)}(T^{-1/3}y) 
%+T^{-1/3} \sum_{n=2}^\infty c_{n}^*\psi_n^{q^{*}(t/T)}(T^{-1/3}y)\Big]
%e^{\gamma'(t)y/\sigma^2(t/T)}\,\dd y \lesssim\\
\E[\#\{u\in N(t): \gamma_T(t)+K - X_u(t) \in \dd y\tand 
X_u(s)\le \gamma_T(s)+K\,,\forall s\le t\}] \\
% \le \Big[C Ke^{y/\sigma(t/T)-K/\sigma(0)}T^{-2/3} 
% \sum_{n=1}^\infty 
% e^{-C'n^{2/3}}\big|\psi_n^{q_{*}(t/T)}(T^{-1/3}y)\big|\Big]\,\dd y
\le 2e^{y/\sigma(t/T)-K/\sigma(0)}G(K,y;t)\,\dd y,
\end{multline*}
% where \(q_{*}(t)\le q_T(t) \) and
% %\le q^*(t)\) with 
% \( q_T(t)-q_{*}(t)\le 2T^{-1/3}\sup_{t\in[0,1]}|q_T'(t)|.\)
where  $G(x,y;t)$ is the fundamental solution to the 
PDE \eqref{eq:pde}, with $Q(t) = q_T(J(t/T))$.
\end{lemma}
\begin{proof}
By a similar argument as the one leading to \eqref{eq:145}, 
the expectation in the statement of the lemma equals
\[
e^{y\gamma_T'(t)/\sigma^2(t/T) - K \gamma_T'(0)/\sigma^2(0) +o(1)} G(K,y;t)\,\dd y.
\]
% where $G(x,y;t)$ is the same as in the proof of
% Lemma \ref{lem:curve}.
By the assumption on $K$, we have $K\gamma_T'(0)/\sigma^2(0) = K/\sigma(0)+o(1)$ and by definition of $\gamma_T$, we have $\gamma_T'(t)\le \sigma(t/T)$. 
The claim follows.
% The claim
% now follows from the analytical
% Proposition~\ref{prop:fund_estimate} and \eqref{eq:green}
% in Appendix A (Section~\ref{app}).
\end{proof}

\section{Tail estimates}
\label{sec-4}
We derive in this section tail estimates on the distribution of $M_T$ summarized in the following proposition.

\begin{proposition}
\label{prop:tail}
There exists a constant $C=C(c_0)$, 
such that for large \(T,\) for any $\sigma\in \Xi_{c_0}$ and
every $K\in[1,T^{1/3}]$,
% There exists a constant $C=C(c_0)$, 
% such that for any $\sigma\in \Xi_{c_0}$, for large $T$ and $K\in[1,T^{1/3}]$,
$$C^{-1} Ke^{-K/\sigma(0)} \le \P(M_T \ge m_T+K) \le CKe^{-K/\sigma(0)}.$$
\end{proposition}

The proof of Proposition~\ref{prop:tail} goes by a suitably truncated first-second moment method, inspired by analogous results in the time-homogeneous case \cite{Bramson78,Aidekon,Roberts,BDZ}. The key ingredients are estimates on a single Brownian particle with time-inhomogeneous variance staying below a curve and reaching a certain point at a given time $t$. These results, which have already been used in the previous section, are obtained in the appendix by analytic methods. However, as in the time-homogeneous case, the first-second moment method applied directly to the particles staying under the curve $\gamma_T$ would not yield the $O(1)$ precision on the maximum at time $T$ that we are aiming at, but would rather induce an error of magnitude $O(\log \log T)$. This can be rectified in our case by slightly changing the curve in the time interval $[T- 
T^{2/3},T]$ in a way similar to the time-homogeneous case (namely, by having it end at the point $\gamma_T(T)-\sigma(1)\log T$. Luckily, for the upper bound it is possible to shortcut this approach, as Slepian's inequality allows us here to directly use existing results in the time-homogeneous case for the system during the time interval $[T- T^{2/3},T]$ (see Section~\ref{sec:upper} for details).

%As mentioned in the introduction, obtaining the existence of a constant $C$, such 

\subsection{Proof of Proposition~\ref{prop:tail}: Upper bound}
\label{sec:upper}
%Set $t_0 = T-T^{2/3}$ and let $x\ge1$. Let $A_1$ be the event that no particle reaches the curve $\gamma_T(t)+x$ until time $t_0$ and let $A_2$ be the event that all particles are below $\gamma_T(t_0)+x-1$ at time $t_0$. Note that by Lemma~\ref{lem:density}, $\P(A_1\cap A_2^c) \le Cxe^{-x/\sigma(0)}/T$. Let $(\F_{t})_{t\ge0}$ be the natural filtration of the BBM. By the branching property and Lemma~\ref{lem:bramson}, we have for large $T$,
Set $t_0 = T- T^{2/3}$ and let $K\ge2$. Let $(\F_{t})_{t\ge0}$ be the natural filtration of the BBM. A union bound gives,
\[
\P(M_{T} \ge m_T + K\,|\,\F_{t_0}) \le \sum_{u\in N(t_0)} \P_{(X_u(t_0),t_0)}(M_{T} \ge m_T + K),
\]
where $\P_{(x,t)}$ denotes the law of BBM with variance $\sigma^2(\cdot/T)$ starting with one particle at the point $x$ at time $t$. 
We will estimate the summands on the right-hand side by comparison with
 a BBM with constant variance. Set $\sigma_c^2 =  
T^{1/3}\int_{1- T^{-1/3}}^1 \sigma^2(t)\,\dd t$. By the assumption on $\sigma$, we have
\[
m_T - \gamma_T(t_0) \ge \ T\int_{1- 
T^{-1/3}}^1\sigma(t)\,\dd t - \sigma(1)\log T -C \ge \sigma_c 
\left( T^{2/3}- \tfrac 3 2 \log T^{2/3} \right)- C_1,
\]
for some constant $C_1$ that we fix for the remainder of this proof. 
Now, let $(Y_u( T^{2/3}))_{u}$ and $(Y^c_u(
T^{2/3}))_{u}$ be the positions of the particles at time $T^{2/3}$ in branching Brownian motions with branching rate 
$\beta_0$ and variances $\sigma^2((\cdot+t_0)/T)$ and $\sigma^2_c$, 
respectively. Conditioned on the genealogy, we have 
$\E[Y_u( T^{2/3})^2] = \E[Y^c_u( T^{2/3})^2]$ and 
$\E[Y_u( T^{2/3})Y_v(
T^{2/3})] \ge \E[Y^c_u( T^{2/3})Y^c_v( T^{2/3})]$ 
for every $u$ and $v$, by the definition of $\sigma_c^2$ and the fact 
that $\sigma^2$ is decreasing. Hence, setting $M^c = 
\max_u Y^c_u( T^{2/3})$, we have by 
Slepian's inequality \cite{slepian} for every $x\ge 1$, 
\[
\P_{(\gamma_T(t_0)+K-C_1-x,t_0)}(M_{T} \ge m_T + K) \le 
\P(M^c\ge\sigma_c \left( T^{2/3}- \tfrac 3 2 \log T^{2/3} \right)+x)
\]
The tail estimates for the maximum of time-homogeneous BBM are available e.g.
in \cite{Bramson83}, and we obtain that
\begin{equation}
  \label{eq-pg7}
  \P_{(\gamma_T(t_0)+K-C_1-x,t_0)}(M_{T} \ge m_T + K) \le 
  Cxe^{-x/\sigma_c}\leq 
  Cxe^{-x/\sigma(t_0/T)}\,,
\end{equation}
  for large $T$, uniformly in $x\ge 1$.
  
Let
$A$ denote the event that no particle reaches the curve 
$\gamma_T(t)+K-C_{1}-1$ until time $t_0$. 
Integrating
%\footnote{We can integrate term by term because $\sum_{n=2}^\infty e^{-C'n^{2/3}}\langle|\psi_n|,x\rangle$ converges by point 4 of Lemma~\ref{lem-bullets}.} 
the upper bound in
Lemma~\ref{lem:density} (taken at time $t=t_0$) against the
distribution in \eqref{eq-pg7}  and using Corollary \ref{cor:G_estimate_2}
now yields for $K\ge 2(C_1+1)$ and large $T$,
%\[
%\P(\{M_{T} \ge m_T + K\}\cap A) \le 
%CKe^{-K/\sigma(0)}\left(\langle \psi_1,x\rangle + 
%T^{-1/3}\sum_{n=2}^\infty e^{-C'n^{2/3}}|\langle \psi_n,x\rangle|\right).
%\]
\[
\P(\{M_{T} \ge m_T + K\}\cap A) \le 
CKe^{-K/\sigma(0)}\,.
\]
The upper bound in the statement of Proposition~\ref{prop:tail} 
now follows from this inequality, together with the fact that 
$\P(A^c) \le CKe^{-K/\sigma(0)}$ for large $T$ by Lemma~\ref{lem:curve}.

\subsection{Proof of Proposition~\ref{prop:tail}: Lower bound}

As discussed above, the proof involves a second moment (``Many-to-two'') argument.
In order to carry it out, we need to modify the curve $\gamma_T(\cdot)$ at
the last interval $[T-T^{2/3},T]$.
Toward this end,
fix $K>1$ and
let $\phi_T(t)$ be an increasing, twice differentiable function\footnote{The 
construction of such a function is possible for large enough $T$, for 
example by gluing together a parabola on $[T-T^{2/3},T-T^{2/3}/2]$ and a 
line on $[T-T^{2/3}/2,T]$.} such that $\phi_T(t) \equiv 0$ on 
$[0,T-T^{2/3}]$, $\phi_T(T)=\sigma(1)\log T$, $\phi_T'(t) 
\le 2\sigma(1)\log T/T^{2/3}$ and $\phi_T''(t) \le 4\sigma(1)\log T/T^{4/3}$. 
Define the curve 
\[\zeta_T(t) = \gamma_T(t)+ K -\phi_T(t).\]
From the definitions, one obtains after some algebraic manipulations
\begin{equation}
\label{eq:ihpouf}
\frac{(\zeta_T'(t))^2}{2\sigma^2(t/T)}=\frac12-T^{-2/3}
\frac{w'(t/T)}{\sigma(t/T)}-\frac{\phi'_T(t)}{\sigma(1)}+o(1/T)\,.
\end{equation}
For $s,t\in [0,T]$, let
\[
G_\zeta(x,y;s,t)\,\dd y = \E_{(K-x,s)}[\#\{u\in N(t):X_u(r)\le \zeta_T(r) 
\forall s\le r\le t,\,
\zeta_T(t) - X_u(t)\in \dd y\}]
\]
denote the expected number of descendants at time $t$
of a particle present at time $s$ 
at
location $K-x$, so that the path of the descendant stayed below the curve 
$\zeta_T(\cdot)$ until time $t$, and 
reached, at time $t$, an infinitesimal neighborhood of
the value $\zeta_T(t)-y$.
Similarly to the proof of \eqref{eq:145}, we have, using 
\eqref{eq:ihpouf}, that
\begin{align}
\nonumber
G_\zeta(x,y;s,t)\,\dd y &= E_{(x,s)}\Big[\exp\Big(\int_s^t
\frac{\zeta_T'(r)}{\sigma^2(r/T)}\,\dd B_r + \frac{t-s}{2}-\int_s^t\frac{(\zeta_T'(r))^2}{2\sigma^2(r/T)}\,\dd r\Big)\Ind_{B_t\in\dd y, H_0(B_{s+\cdot})>t-s}\Big]\\
&= \exp\left(\frac{\zeta_T'(t)}{\sigma^2(t/T)}y - \frac{\zeta_T'(s)}{\sigma^2(s/T)}x + \frac{\phi_T(t)-\phi_T(s)}{\sigma(1)} + o(1)\right) G(x,y;s,t)\,\dd y,
\end{align}
where under $P_{(x,s)}$, $(B_t)_{t\ge s}$ is the time-inhomogeneous
Brownian motion starting at time $s$ at $x$ and with instantaneous variance
%with time-inhomogeneous variance 
$\sigma^2(\cdot/T)$ and
$G(x,y;s,t)$ is the fundamental solution to \eqref{eq:pde}, 
with $Q(t) = |\sigma'(J(t/T))|/\sigma^2(J(t/T))+O(\log T/T^{2/3})$.
%(The $o(1)$ term in the last display comes from the time-inhomogeneity in the
%quadratic term of Girsanov's theorem.) 
In particular, if $N_T$ denotes the number of particles, at time $T$,
whose trajectory stayed under the curve $\zeta_T(\cdot)$ and reached
the interval $[\zeta_T(T)-2,\zeta_T(T)-1]$ at time $T$, then, for large $T$,
%that stayed under the curve $\zeta_T(t)$ until time $T$, then for large $T$,
\begin{equation}
 \label{eq:N_T}
\E[N_T] = \int_1^2 G_\zeta(0,y;0,T)\,\dd y \ge C Te^{-K/\sigma(0)}\int_1^2 
G(K,y;0,T)\,\dd y \ge C K e^{-K/\sigma(0)},
\end{equation}
where the last inequality follows from  \eqref{eq:ihp1} of 
Corollary \ref{cor:G_estimate_1}.
%Proposition~\ref{prop:fund_estimate} and \eqref{eq:green}.

We now estimate the second moment of $N_T$. For the rest of the proof, fix the constant $C_1 =  1/(2\sigma(0))$, which satisfies $C_1 \le \inf_{t\in[0,T]}\zeta'_T(t)/\sigma^2(t/T)$ for large $T$. The second moment formula\footnote{It can be derived 
%heuristically 
by conditioning on the splitting time of pairs of particles.} (``Many-to-two lemma'') for branching Markov processes \cite[Theorem~4.15]{INW3} then yields for large $T$,
\begin{align}
\nonumber
&\E[N_T^2] = \E[N_T] + \beta_0\E_\mu[L^2-L] \int_0^T\,\dd t \int_0^\infty \,\dd y\, G_\zeta(K,y;0,t) \left(\int_1^2 G_\zeta(y,z;t,T)\,\dd z\right)^2\\
 &\le \E[N_T] + C e^{-K/\sigma(0)} \int_0^T T\,\dd t \int_0^\infty \,\dd y\, G(K,y;0,t) \left(\int_1^2 G(y,z;t,T)\,\dd z\right)^2e^{-C_1y+\frac{\phi_T(T)-\phi_T(t)}{\sigma(1)}}.
\end{align}
We split the integral into three parts, according to intervals of time 
$[0,\kappa T^{2/3}]$, $[\kappa T^{2/3},T-\kappa T^{2/3}]$ and $[T-\kappa T^{2/3},T]$ and denote the three parts by $I_1$, $I_2$ and $I_3$. 
In order to estimate the first and third part, we bound the Green kernel 
$G(x,y;s,t)$ for $t-s\le \kappa T^{2/3}$ by the Green kernel of Brownian motion killed at the origin. Namely, 
writing $V(t) = \int_0^t\sigma^2(s/T)\,\dd s$, we have for  $t-s\le 
\kappa T^{2/3}$ and $x,y\ge0$,
\begin{equation}
G(x,y;s,t) \le \frac{C}{\sqrt{t-s}}\exp\left(-\frac{(x-y)^2}{2(V(t)-V(s))}\right)\left(\frac{xy}{t-s}\wedge 1\right)
\label{eq:green_killed}
\end{equation}
For $t\ge \kappa T^{2/3}$, we use 
Corollary \ref{cor:G_estimate_1}
%Proposition~\ref{prop:fund_estimate} and \eqref{eq:green} 
in order to  bound $G(K,y;0,t)$ and $G(y,z;T-t,T)$ (for the latter, we consider the time-reversal of \eqref{eq:pde}, and $Q(\cdot)$ as above). 
%Note that $|\psi_n^q(x)|\le \sqrt q x$ for every $n,q,x$. 
%In order to make the formulae easier to read, we will only write the first term in the expansion of Proposition~\ref{prop:fund_estimate}, it is easy to check that the other terms to not harm. 
This yields $G(K,y;0,t)\le CT^{-1}Ky$ and $G(y,z;T-t,T)\le CT^{-1} y$ 
for every $t\ge \kappa T^{2/3}$ and $z\in[1,2]$.

For the first part, we now get by exchanging integrals, 
\[
I_1 \le T^2 \int_0^\infty (T^{-1}y)^2 e^{-C_1 y}\left(\int_0^{\kappa 
T^{2/3}}G(K,y;0,t)\,\dd t\right) \,\dd y \le C \int_0^\infty y^2 e^{-C_1 y} (1+Ky) \,\dd y \le CK,
\]
for $K\ge 1$ and large $T$. Here, we used the fact that by \eqref{eq:green_killed},

\[
\int_0^{\kappa T^{2/3}}G(K,y;0,t)\,\dd t \le C \int_0^1 \frac{1}{\sqrt t}\,\dd t+C\int_1^\infty\frac{Ky}{t^{3/2}}\,\dd t \le C(1+Ky).
\]
For the second part, we have
\[
I_2 \le C T^2 \int_{\kappa T^{2/3}}^{T-\kappa T^{2/3}} \,\dd t \int_0^\infty T^{-3} K y^3 e^{-C_1y}\,\dd y \le CK,
\]
For the third part, we note that by \eqref{eq:green_killed} and the assumptions on $\phi_T$, we have for every $y\ge 0$, for large $T$,
\begin{multline*}
\int_1^{\kappa T^{2/3}} \left(\int_1^2 G(y,z;T-t,T)\,\dd z\right)^2e^{\frac{\phi_T(T)-\phi_T(T-t)}{\sigma(1)}}\,\dd t \\
\le Cy^2\left(\int_1^{T^{2/3}/\log T}t^{-3}\,\dd t + T^{2/3} T \left(\frac{T^{2/3}}{\log T}\right)^{-3}\right) \le Cy^2.
\end{multline*}
Furthermore, for $t\le1$, we have $\left(\int_1^2 G(y,z;T-t,T)\,\dd z\right)^2\exp((\phi_T(T)-\phi_T(T-t))/\sigma(1))\le C$ for every $y$. This gives,
\[
I_3 \le \int_0^\infty K y^3 e^{-C_1 y}\,\dd y + \int_0^1 CK\,\dd t \le CK.
\]
In total, we have
\[
\E[N_T^2] \le  \E[N_T] + C e^{-K/\sigma(0)}(I_1+I_2+I_3) \le C\E[N_T],
\]
by \eqref{eq:N_T}. This now yields,
\[
\P(N_T \ge 1) \ge \frac{\E[N_T]^2}{\E[N_T^2]} \ge C^{-1}\E[N_T].
\]
Together with \eqref{eq:N_T}, this finishes the proof of the lower bound in Proposition~\ref{prop:tail}.

\section{Proof of Theorem~\ref{th:1}}
\label{sec-5}

Armed with the tail estimates provided by Proposition~\ref{prop:tail}, the proof of Theorem~\ref{th:1} follows by considering the descendants of the particles living at a large (but fixed) time $t$. Here are the details: 

We assume without loss of generality that $\sigma(0) = 1$ (otherwise we 
can 
rescale space). Write $\P^T$ and $\E^T$ in place of $\P$ and $\E$, similarly, we write $\P^T_{(x,t)}$ in place of $\P_{(x,t)}$ (see Section~\ref{sec:upper}). Furthermore, we will denote by $\P_{\text{hom}}$ and $\E_{\text{hom}}$ the law of (time-homogeneous) branching Brownian motion with variance $1$ and branching rate $\beta_0$, starting with one particle at the origin. In what follows, we fix $y\in\R$ and let $t\ge0$ large enough, such that $|y|<\log t -2$. We will later let  first $T$, then $t$ go to infinity, i.e.\ we will choose $t$ as a function of $T$, such that $t(T)$ goes to infinity slowly enough as $T\to\infty$. 

As in Section~\ref{sec:upper}, let $(\F_{t'})_{t'\ge0}$ be the natural filtration of the BBM.
Define the $\F_t$-measurable random variable $W_{t,T}$ by
\[
W_{t,T} = \P^T(M_T\le m_T+y\,|\,\F_t) = \prod_{u\in N(t)} \left(1-\P^T_{(X_u(t),t)}(M_T\ge m_T+y)\right).
\]
Furthermore, define 
$$D_t = \sum_{u\in N(t)} (t-X_u(t))e^{X_u(t)-t}\,.$$
By Proposition~\ref{prop:tail} applied with the function
$\bar\sigma(t')=\sigma((t'(T-t)+t)/T)$, there exists a constant $C$ and for each large $T$ a function $g_{t,T}:\R_+\to[C^{-1},C]$, such that for each $x\in[-t,t -\log t]$,
\begin{equation}
  \label{eq-exp33}
1-\P^T_{(x,t)}(M_T\ge m_T+y) = \exp\left(-g_{t,T}((y-x+t)/\sqrt t) (y-x+t) e^{-(y-x+t)}\right).
\end{equation}
By the continuity of $\P^T_{(x,t)}$ in $x$, the functions $g_{t,T}$ are actually continuous, in particular, they are Lebesgue-measurable. 

As in Section~\ref{sec:derivative_Gibbs} (note that if $(B_t)_{t\ge 0}$ is a Brownian motion started at the origin, then $(t-B_t)_{t\ge 0}$ is a Brownian motion with drift $+1$ started at the origin), define the derivative Gibbs measure
\[
\mu_{t} = \sum_{u\in\mathcal N(t)} (t - X_u(t))e^{-(t - X_u(t))}\delta_{(t-X_u(t))/\sqrt t}\ .
\]
Then, on the event
$A_t = \{\forall u\in N(t):-t\le X_u(t)\le t-\log t\},$
we get by \eqref{eq-exp33}
\begin{equation}
\label{eq:W}
W_{t,T} \Ind_{A_t}= \exp\left(-e^{-y}\int_0^\infty g_{t,T}(y/\sqrt t+x)\mu_{t}(\dd x)\right)\Ind_{A_t}
\end{equation}
and $\P_{\text{hom}}(A_t)\to 1$ as $t$ goes to infinity \cite{Bramson83}.
Now, note that as $T\to\infty$, the law of the process until time $t$ converges to its law under $\P_{\text{hom}}$, 
because conditioned on the genealogical structure and the branching times, 
the particle motion until time $t$ on each of the finitely many branches of 
the genealogical tree converges to Brownian motion with variance $1$. 
Moreover, thanks to the continuity and positivity of the Gaussian density, 
we can construct a probability space with probability measure 
$\widetilde \P$ which supports random variables $(\widetilde \mu_T)_{T\ge 0}$ 
and $\widetilde \mu$, such that, under $\widetilde \P$, $\widetilde \mu_T$ 
follows the law of $\mu_t$ under $\P^T$, $\widetilde \mu$ follows the law of 
$\mu_t$ under $\P_{\text{hom}}$ and $\widetilde \mu_T = \widetilde \mu$ 
on an event $\widetilde G_T$ with $\widetilde \P(\widetilde G_T)\to 1$ as 
$T\to\infty$. In particular,
\begin{equation}
  \label{eq-exp1}
\int_0^\infty g_{t,T}(y/\sqrt t+x)\widetilde \mu_T(\dd x) = \int_0^\infty g_{t,T}(y/\sqrt t+x)\widetilde \mu(\dd x)\quad\text{on $\widetilde G_T$, for every $y$.}
\end{equation}
By a diagonalization argument, we can now choose $t=t(T)$ 
growing slowly with
$T$, so that \eqref{eq-exp1} continues to hold with this choice of $t(T)$. 
%Letting now $t$ go to infinity slowly enough with $T$, we get 
By Theorem~\ref{th:Gibbs}, we have that
for every bounded continuous function $f$,
\begin{align*}
&\E_{\text{hom}}\left[f\left(
\int_0^\infty g_{t(T),T}(y/\sqrt{ t(T)}+x)\mu_{t(T)}(\dd x)\right)\right] 
\\
&-
\E_{\text{hom}}\left[f\left(
D_\infty 
\int g_{t(T),T}(y/\sqrt{ t(T)}+x)\rho(\dd x)\right)\right] \to_{T\to\infty} 0,
\end{align*}
where $\rho$ is the law of a BES(3) process at time 1, started at 0, 
and the variable $D_\infty$ is the derivative martingale limit from 
Section~\ref{sec:derivative_Gibbs}. 
Using the above coupling we conclude that
\begin{align}
\label{eq:riu}
&\E^T\left[f\left(\int_0^\infty g_{t(T),T}(y/\sqrt{ t(T)}+x)\mu_{t(T)}(\dd x)
\right)\right] 
\nonumber\\
&- \E_{\text{hom}}\left[f\left(D_\infty 
\int g_{t(T),T}(y/\sqrt{ t(T)}+x)\rho(\dd x)\right)\right] \to 0.
\end{align}
On the other hand,
since $\rho$ has a continuous density with respect to Lebesgue measure, 
we have,
\begin{equation}
\label{eq:398}
%\lim_{t\to\infty} 
\limsup_{T\to\infty}\left|\int g_{t(T),T}(y/\sqrt{ t(T)}
+x)\rho(\dd x) - \int g_{t(T),T}(x)\rho(\dd x)\right| = 0
\end{equation}
Setting $C_T = \int g_{t(T),T}(x)\rho(\dd x)$, we get 
by \eqref{eq:W}, \eqref{eq:riu}, \eqref{eq:398} and dominated convergence,
\[
\lim_{T\to\infty} \P^T(M_T\le m_T+y-\log C_T) = \E_{\text{hom}}[e^{-e^{-y} D_\infty}] = \phi(x),
\]
where $\phi$ is a solution to \eqref{eq:fkpp}, see Section~\ref{sec:derivative_Gibbs}.
This yields Theorem~\ref{th:1}.

\noindent
{\bf Remark:} While a-priori, the constant $C_T$ depends on the particular
choice of sequence $t(T)$, it is clear that  the conclusion of Theorem 
\ref{th:1} implies that a-posteriori, it is independent of this 
choice.

%{\tiny
% By the tower property of conditional expection and by dominated convergence, we now get
%\[
%\limsup_{T\to\infty} \P^T(M_T\le m_T+y) \le \limsup_{t\to\infty}\limsup_{T\to\infty}\E^T[W_{t,T}\Ind_{A_t}]+\P^T(A_t)\le \limsup_{t\to\infty}\E_{\text{hom}}[e^{-C e^{-y} D_t}],
%\]
%and
%\[
%\liminf_{T\to\infty} \P^T(M_T\le m_T+y) \ge \liminf_{t\to\infty}\liminf_{T\to\infty}\E^T[W_{t,T}\Ind_{A_t}] \ge \liminf_{t\to\infty}\E_{\text{hom}}[e^{-C^{-1} e^{-y} D_t}].
%\]

%Now, under $\P_{\text{hom}}$, the process $(D_t)_{t\ge0}$ is the so-called  \emph{derivative martingale} and it is known \cite{LS87,Neveu} that it converges almost surely as $t\to\infty$ to a non-degenerate random variable $D_\infty$, whose Laplace transform is given by $\E_{\text{hom}}[\exp(-e^{-x}D_\infty)] = \phi(x)$, where $\phi$ is a solution to \eqref{eq:fkpp}. Since $\phi(\cdot+C)$ is also a solution to \eqref{eq:fkpp} for every $C$, this yields Theorem~\ref{th:1}.
%}

%\bibliography{bbm_inhomogeneous}
\appendix
\section{An Airy-type PDE with time-varying parameters}
\label{app}

We are interested in the following parabolic PDE:
\begin{equation}
\label{eq:pde_canonical}
w_t = \ep^{-1}\big\{w_{xx} - q(t)xw\big\},\quad w(t,0) = 0\ \forall t\ge0,
%w_t = A\big\{w_{xx} - (q(t)+A^{-2}r(t))xw\big\},\quad w(0,x) = \delta(x-x_0),\quad w(t,0) = 0\ \forall t\ge0,
\end{equation}
for $q\in C^{1}[0,1],$ \(q>0\). We want to study its behaviour as $\ep\to0$.

 Before solving this equation, we recall some facts about the \emph{Airy differential operator} $L\psi = \psi'' - x\psi$. Let $L^2([0,\infty))$ be the space of square-integrable functions on $[0,\infty)$ and let $\langle\cdot,\cdot\rangle$ be the associated scalar product\footnote{We will also use the notation $\langle f,g\rangle =\int_0^\infty f(x)g(x)\,\dd x$ for $g\not\in L^2([0,\infty))$, as long as the integral is well defined.} with norm $\|\cdot\|_{2}$. 
Recall the definition \eqref{eq-Airy} of the \emph{Airy function of the first kind} $\Ai(x)$. We denote
 %(Lebedev, Section 5.17 or Abramowitz, Stegun, Section 10.4) 
 by $-\alpha_1 > -\alpha_2 >\cdots$ its discrete set of zeros, 
 with $\alpha_1 = 2.33811...$ . The functions $\psi_n$ defined by
\[
\psi_n(x) = \frac{\Ai(x-\alpha_n)}{\|\Ai(\cdot-\alpha_n)\|_{2}},\quad n=1,2,\ldots
\] 
then form an ONB of $L^2([0,\infty))$ and $\psi_n$ is an eigenfunction of $L$ with eigenvalue $-\alpha_n$ \cite[Section~4.4]{Vallee}.

The following lemma collects some other facts about the functions $\psi_n(x)$, which are probably well-known, although we could not find a reference to some of them.
\begin{lemma}$ $
  \label{lem-bullets}
  \begin{enumerate}
  \item $\|\Ai(\cdot-\alpha_n)\|_{2} = |\Ai'(-\alpha_n)|$ for all $n$. In particular, $\psi_n'(0)=1$ for all $n$.
%  \item If $-\alpha'_n$ is the $n$-th zero of $\Ai'$, then $\alpha'_n < \alpha_{n+1} < \alpha'_{n+1}$ for all $n\ge 1$. Furthermore, $\alpha'_1 > 0$.
\item $\alpha_nn^{-2/3}\to 3\pi/2$ as $n\to\infty$.
        \item  $|\psi_n(x)| \le x$ for all $n\ge 1$ and $x\ge 0$.%\footnote{To see this, first note that by the previous equality, $|\Ai'(-\alpha_n)|$ is increasing in $n$. Furthermore, $|\Ai'(-\alpha_n)|\ge |\Ai'(x)|$ for all $x\in[-\alpha'_n,-\alpha'_{n-1}]$ (with $-\alpha'_0 = +\infty$), since  $\Ai''(x)\Ai(x)<0$ for $x<0$, as can be seen directly from the Airy differential equation. It follows that $|\Ai'(x)|\le|\Ai'(-\alpha_n)|$ for all $x\ge -\alpha_n$, which yields the claim.}
  \item For some numerical constant $C$, $\langle|\psi_n|,x\rangle \le Cn^{4/3}$ for all $n\ge 1$.
\end{enumerate}
\end{lemma}
\begin{proof}
The first and second points are \cite[(4.52) and (2.52)]{Vallee}, respectively. For the third point, we first note that since $\Ai''(x) = x\Ai(x)$, the local extrema of $\Ai'$ on $\R$ are exactly the zeros of $\Ai$ and the origin. Furthermore, by the first point of the lemma, $|\Ai'(-\alpha_n)|$ is increasing in $n$ and by \cite[(3.50)]{Vallee}, $|\Ai'(0)| < |\Ai'(-\alpha_1)|$. This yields $|\Ai'(x)|\le|\Ai'(-\alpha_n)|$ for all $x \ge -\alpha_n$, from which the third point of the lemma follows. 

The third point of the lemma in particular implies $\langle|\psi_n|\Ind_{x\le \alpha_n},x\rangle \le \alpha_n^2/2$ for all $n$. Now,
$$\langle |\psi_n|\Ind_{x\ge \alpha_n},x\rangle = \|\Ai(\cdot-\alpha_n)\|_{2}^{-1}\langle |\Ai|,x+\alpha_n\rangle \le \|\Ai\|_{2}^{-1}\langle |\Ai|,x+\alpha_n\rangle.$$
By the tail bound $\Ai(x)\le \exp(-(2/3)x^{3/2})$ for large $x$ \cite[10.4.59]{AbrSteg}, the expression on the right-hand side of the last inequality is finite, whence $\langle |\psi_n|\Ind_{x\ge \alpha_n},x\rangle \le  C\alpha_n$, for some numerical constant $C$.
% Together with 
% %the first point of the lemma and 
% the definition of $\psi_n$, it follows that $$\langle |\psi_n|\Ind_{x\ge \alpha_n},x\rangle \le (\|\Ai(\cdot-\alpha_n)\|_{2})^{-1}\langle |\Ai|,x+\alpha_n\rangle \le C\alpha_n$$, for some numerical constant $C$. 
%Furthermore, by the third point of the lemma, we have $\langle|\psi_n|\Ind_{x\le \alpha_n},x\rangle \le \alpha_n^2/2$ for all $n$. 
Applying the second point of the lemma shows the fourth point.
\end{proof}

We get back to the equation \eqref{eq:pde_canonical}. Define for a constant $q$ the operator $L_qu = u_{xx} - qxu$. One easily checks that the function $\psi^q_n(x) = q^{1/6}\psi_n(q^{1/3}x)$ is an eigenfunction of $L_q$ with eigenvalue $-\alpha_nq^{2/3}$ and the functions $\psi^q_n$ form an ONB of $L^2([0,\infty))$. We further denote by \( g(x,y;t):=g(x,y;0,t)\)  
the fundamental solution of \eqref{eq:pde_canonical}.
\begin{proposition}
\label{prop:fund_estimate}
Set $Q_1 = \inf_{t\in[0,1]} q(t)^{2/3}$ and $Q_2 = \sup_{t\in[0,1]}|(\log q)'(t)|$.  Suppose $Q_1 > 0$. 
Then there exists \(C_0=C_0(Q_{1}^{-1},Q_{2})>0\)  depending continuously on its parameters, such that for all  $\ep>0$, \(t\in[4\varepsilon,1]\) and $\delta \in [\ep,\sqrt\ep]$ there exist 
\begin{itemize}
 \item \((c_{*n})_{n\geq2}\) and \((c_{n}^{*})_{n\geq2}\)  with  \(|c_{*n}|\vee |c_{n}^{*}|\leq C_0\exp(-C_0^{-1}(t \wedge \delta)\ep^{-1} n^{2/3})\), 
 \item \(q_{*}(t)\le q(t) \le q^*(t)\) with \( q^*(t)-q_{*}(t)\le 2(t\wedge \delta)^2\ep^{-1} \sup_{t\in[0,1]}|q'(t)|,\) 
 \item $C_1=C_1(\ep,\delta,C_0)>1$ depending continuously on its parameters and satisfying $C_1\to1$ as $\ep\to 0$ and $\delta/\sqrt{\ep}\to 0$,
\end{itemize}
such that for all $x\in[0,1]$,  \[C_1^{-1}\left(\psi_1^{q^*(t)} +\varepsilon \sum_{n=2}^\infty c_n^*\psi_n^{q^*(t)}\right)\le \frac{ g(x,\cdot;t)}{\psi_1^{q(0)}(x)}\exp\left(\ep^{-1} \alpha_1 \int_0^t q(s)^{2/3}\,\dd s\right) \le C_1\left(\psi_1^{q_{*}(t)} +\varepsilon \sum_{n=2}^\infty c_{*n}\psi_n^{q_{*}(t)}\right).\]
\end{proposition}

Before providing the proof of Proposition \ref{prop:fund_estimate}, we derive
some a-priori estimates on solutions of \eqref{eq:pde_canonical}.
\begin{lemma}
\label{prop:pde}
Define $Q_1$ and $Q_2$ as in Proposition~\ref{prop:fund_estimate} and assume $Q_1>0$. Let $w(t,x)$ be the solution to \eqref{eq:pde_canonical} with initial condition satisfying $\|w(0,\cdot)\|_2 \le 1$. Define for each $t\ge0$ the function $W_t(x) = \exp(\int_0^t \ep^{-1} \alpha_1 q(s)^{2/3}\,\dd s)w(t,x).$  Then there exist numerical constants $C,C_{1}$, such that  for all \(t\in[0,1]\), 
\begin{enumerate}
        \item $\|W_t\|_2 \le 1$,
        \item $|\langle W_t,\psi^{q(t)}_1\rangle - \langle W_0,\psi^{q(0)}_1\rangle| \le C\frac{Q_2+1}{Q_1}\ep$ and
        \item $\left(\sum_{n\ge 2} \langle W_t,\psi^{q(t)}_n\rangle^2\right)^{1/2} \le C\frac{Q_2+1}{Q_1}\ep\left(\frac{Q_2+1}{Q_1}\ep+|\langle W_0,\psi^{q(0)}_1\rangle|\right)+\exp(-C_1\ep^{-1}Q_1 t)$.
\end{enumerate}
\end{lemma}

\begin{proof}
After decomposing the solution of
\eqref{eq:pde_canonical}
in the eigen-basis determined by the Airy functions, the proof
proceeds by analyzing a coupled system of linear, time inhomogeneous, ordinary
differential equations.

Throughout the proof, $C$, $C_1$ and $C_2$ are some numerical constants which may change from line to line. Define the vector $\mathbf c(t) = (c_1(t),c_2(t),\ldots)^T$, where $c_n(t) = \langle W_t,\psi^{q(t)}_n\rangle$. From \eqref{eq:pde_canonical}, one gets 
\[
\dot c_n(t) =  -\ep^{-1}(\alpha_n-\alpha_1) q(t)^{2/3}c_n(t) + \sum_{k\ge 1} c_k(t) q'(t) \langle \psi^{q(t)}_k, \frac{\dd}{\dd \tilde q} \psi^{\tilde q}_n\Big|_{\tilde q=q(t)}\rangle,
\]
whence 
\begin{equation}
\label{eq:ode_c}
\dot {\mathbf c}(t) = (D(t) + A(t))\mathbf c(t),\quad D(t) = -\ep^{-1}q(t)^{2/3} \operatorname{diag}(\alpha_i-\alpha_1)_{i\ge 1},\quad A(t) = (\log q)'(t) A.
\end{equation}
Here, $A$ is the antisymmetric matrix
\[
A = \tfrac 1 6 \left(I + 2( \langle x\psi_i',\psi_j\rangle )_{i,j\ge 1}\right) = \tfrac 1 6 ( \langle x\psi_i',\psi_j\rangle - \langle x \psi_j',\psi_i\rangle )_{i,j\ge 1},
\]
where the equality is easily verified by integration by parts\footnote{In fact, $A_{ij} = 2 (-1)^{i+j}(\alpha_i - \alpha_j)^{-3}$ for $i\ne j$, which can be easily verified by the equation \cite[(3.54)]{Vallee}, however, we will not use this fact.}. 

Since $D(t)+A(t)$ and $D(t')+A(t')$ do not commute unless $q(t)^{2/3} (\log q)'(t') = q(t')^{2/3} (\log q)'(t)$, there is no obvious explicit expression for the solution to \eqref{eq:ode_c}. However, since $D$ is diagonal and $A$ antisymmetric, we have
\[
\frac{\dd}{\dd t} \|\mathbf c(t)\|^2_2 = \frac{\dd}{\dd t} \mathbf c^T(t) \mathbf c(t) = \mathbf c^T(t)(D^T(t)+A^T(t)+D(t)+A(t)) \mathbf c(t) = 2 \mathbf c^T(t) D(t) \mathbf c(t) \le 0,
\]
 by the positivity of $q(t)$. This implies the first claim. In particular, $|c_1(t)| \le 1$ for all $t\ge 0$. Setting $\bar{\mathbf c}(t) = (0,c_2(t),c_3(t),\ldots)^T$, the previous equation yields,
\begin{align*}
\frac{\dd}{\dd t} \bar{\mathbf c}^T(t) \bar{\mathbf c}(t) &= 2 \bar{\mathbf c}^T(t) D(t) \bar{\mathbf c}(t) - 2c_1(t)\sum_{j=2}^\infty A_{1j}(t)c_j(t)\\
&\le -2\ep^{-1}q(t)^{2/3}(\alpha_2-\alpha_1) \bar{\mathbf c}^T(t ) \bar{\mathbf c}(t) + 2|c_1(t)|\,\|(A_{1j}(t))_{j\ge 2}\|_2 \|\bar{\mathbf c}(t)\|_2,
\end{align*}
by the Cauchy--Schwarz inequality. By Parseval's formula, $\|(A_{1j})_{j\ge 2}\|_2 \le \|x\psi_1'\|_2/3<\infty$. This yields
\[
\frac{\dd}{\dd t} \|\bar{\mathbf c}(t)\|_2 \le -C_1\ep^{-1}Q_1\|\bar{\mathbf c}(t)\|_2 + C_2 Q_2|c_1(t)|.
\]
Note that the general solution to the equation \(f'(t)=-af(t)+b\) is \(f(t)=(b/a)+\widetilde Ce^{-at},\) for arbitrary $\widetilde C\in\R$. Since $\bar{\mathbf c}(0) \le 1$, Gr\"onwall's inequality 
now yields that 
\begin{equation}
\label{eq:387}
\|\bar{\mathbf c}(t)\|_2 \le C(Q_2/Q_1)\ep\sup_{s\in[0,1]}|c_{1}(s)| + \exp(-C_1\ep^{-1}Q_1 t),
\end{equation}
 In order to show the second claim, we note that by \eqref{eq:ode_c},
for every \(t\in[0,1],\)\[
\left|c_1(t)-c_1(0)\right| \le \int_0^t \left|\sum_{j=2}^\infty A_{1j}(t)c_j(t)\right|\,\dd t \le C \int_0^t \|\bar{\mathbf c}(t)\|_2\,\dd t,
\]
where the last inequality follows from the Cauchy--Schwarz inequality as above. Together with \eqref{eq:387} and the fact that \(\sup_{t\in[0,1]}|c_{1}(t)|\leq1\), this implies the second claim.
The third claim follows from this, together with  \eqref{eq:387}.
\end{proof}

%\begin{remark}
%We can even show  \[\left(\sum_{n\ge 2} \langle W_1,\psi^{q(1)}_n\rangle^2\right)^{1/2} \le C(\max(Q_2,1)/Q_1)\ep\max(\langle %W_0,\psi^{q(0)}_1\rangle, \max(Q_2,1)/Q_1)\ep).\] Either by refining the proof with the new $c_1(t)$ estimate or by cutting %the time interval $[0,1]$ into two pieces: $[0,1/2]$ and $[1/2,1]$. In the first interval, plainly use the proposition, %then renormalize such that $\|W_{1/2}\|_2 = 1$, then use this as new initial condition for the next piece.
%\end{remark}
%

\begin{proof}[Proof of Proposition~\ref{prop:fund_estimate}]
Fix \(t\in[4\varepsilon,1]\) and $\delta\in[\ep,t-3\ep]$. We can construct $q^*,q_*\in C^1([0,1])$ such that the following holds:
\begin{itemize}
		%[nolistsep,label=$-$]
\item $Q_{1}\le\ q_*\le q \le q^*$
\item \(q_*\equiv q\equiv q^{*}\) on \([2\varepsilon,t-\varepsilon-\delta],\)
\item \(q^{*}\) and \(q_{*}\) are constant on \([0,\varepsilon]\cup[t-\delta,t]\),
\item \(\sup_{s\in[0,1]}\max\{|(\log q_{*})'(s)|,|(\log q^{*})'(s)|\}\leq Q_{2}\) and
\item \(q^{*}-q_{*}\leq2(t\wedge\delta)^2\varepsilon^{-1} \sup_{t\in[0,1]}|q'(t)|\).
\end{itemize}
 Now let $x\in[0,1]$. Let $w^*$ and $w_*$ denote the solutions to \eqref{eq:pde_canonical} with $q$ replaced by $q^*$ or $q_*$, respectively, and with initial condition \(w^{*}(0,\cdot) = w_{*}(0,\cdot) = \delta(\cdot-x)\). By the parabolic maximum principle \cite[Theorem~7.1.9]{Evans}, we then have $w^{*}(t',y)\le g(x,y;t')\le w_*(t',y)$ for all $y\ge0$ and $t'\in[0,1]$.

Write $W^*_{t'}(y) = w^*(t',y)\exp(\ep^{-1} \alpha_1\int_0^{t'}  q^*(s)^{2/3}\,\dd s)$ for all $t',y$.  For every $n$, we have by the first point of Lemma~\ref{lem-bullets} and the fact that $\psi_1(x)>0$ for all $x>0$,\[|\langle W^{*}_0,\psi_n^{q^*(0)}\rangle|= |\psi_n^{q^*(0)}(x)| \le C\psi_1^{q(0)}(x),\] for some constant \(C\) depending on \(Q_{1}\). By \eqref{eq:ode_c}, we then have \[|\langle W^{*}_\ep,\psi_n^{q^*(0)}\rangle|\le C\exp(-(\alpha_n-\alpha_1)q^*(0)^{2/3})\psi_1^{q(0)}(x),\]
for every $n$, since the off-diagonal terms cancel by the fact that $q^*$ is constant on $[0,\ep]$. Together with the second point of Lemma~\ref{lem-bullets}, this yields \(\|W^{*}_\ep\|_2 \le C_1\) for some constant $C_1$ as $\ep$ is small enough. Furthermore, 
\[\langle W^{*}_\ep,\psi_1^{q^*(0)}\rangle = \langle W^{*}_0,\psi_1^{q^*(0)}\rangle=(1+o(1))\psi_1^{q(0)}(x),\] 
where $o(1)$ is a term depending on $Q_1$ and $Q_2$ which vanishes as $\ep \to 0$. Applying Lemma~\ref{prop:pde} with initial condition $w(0,\cdot) =  W^{*}_\ep/C_1$, we get that \(\langle W^{*}_{t},\psi_1^{q^*(t)}\rangle=(1+o(1))\psi_1^{q(0)}(x)\) and $\left(\sum_{n\ge 2} \langle W^*_{t-\delta},\psi^{q^*(t)}_n\rangle^2\right)^{1/2} \le C_2\ep \psi_1^{q(0)}(x)$ for small $\ep$, where $C_2$ depends on $Q_1$ and $Q_2$. As above, this now implies that for every $n\ge 2$, for small $\ep$,
\[
|\langle W^*_t,\psi_n^{q^*(0)}\rangle| \leq\exp(-CQ_1(t\wedge\delta)\ep^{-1}n^{2/3})C_2\ep \psi_1^{q(0)}(x). 
\]
Together with the previous estimates, this finally yields the existence of a sequence of constants \((c_{n}^{*})_{n\geq2}\) with \(|c_{n}^{*}|\leq\exp(-C_{}Q_1(t\wedge\delta)\ep^{-1}n^{2/3})C_2,\) such that as \(\ep\to0,\) 
\[
w^*(t,\cdot) \geq (1+o(1))\exp\left(-\ep^{-1} \alpha_1 \int_0^t q^*(s)^{2/3}\,\dd s\right) \psi_1^{q(0)}(x)\left(\psi_1^{q^*(t)} +\varepsilon \sum_{n=2}^\infty c_n^*\psi_n^{q^*(t)}\right).
\]
An analogous formula holds for $w_*.$ The statement now follows from the fact that $\int_0^1 q^*(s)^{2/3}-q_*(s)^{2/3}\,\dd s = O((t\wedge\delta)^{2})$ by construction.
\end{proof}

Fix \(T>0.\) 
%In the application in Section \ref{sec-3}, we need to consider 
%PDE's on the time interval $[0,T]$. 
The results obtained in the current
section can be easily transported to the  following PDE on $[0,T]\times \R_+$, encountered in Sections~\ref{sec-3} and \ref{sec-4}.
 \begin{equation}
 \label{eq:pde}
  u_t(t,x) = \tfrac 1 2 \sigma^2(t/T)  u_{xx}(t,x) +\{- T^{-1}Q(t)x +T^{-2/3}\alpha_1Q(t)^{2/3}\left( \tfrac 1 2 \sigma^2(t/T) \right)^{1/3}\} u(t,x),
 \end{equation}
 with Dirichlet boundary condition at 0 and where $Q\in C^1([0,T])$ with $Q(t)>0$ for all $t\in[0,T]$.
Setting $J(t) = \int_0^t \tfrac 1 2 \sigma(s)^2\,\dd s$  as in 
\eqref{eq-J}, defining $q(t)$ by \(q(J(t)/J(1)) = 2Q(t)/\sigma^2(t)\), 
and changing variables by 
\[
u(t,x) = w(J(t/T)/J(1),T^{-1/3}x)\exp\left(J(1)T^{1/3}\alpha_1\int_0^{J(t/T)/J(1)}q(s)^{2/3}\,\dd s\right),
\] we see that the function $w(t,x)$ solves  \eqref{eq:pde_canonical} on \([0,1]\times\R_+\) with \(\ep^{-1} = J(1)T^{1/3}\) and the $q(t)$ defined here.
In particular, if $G(x,y;t):=G(x,y; 0,t)$ and \(g(x,y;t)=g(x,y;0,t) \) denote the fundamental solutions of \eqref{eq:pde}
 and \eqref{eq:pde_canonical}, respectively, then we have the relation
 \begin{multline}
 \label{eq:green}
 G(x,y;t) = T^{-1/3}g\left(T^{-1/3}x,T^{-1/3}y;J(t/T)/J(1)\right)\\
 \times\exp\left(J(1)T^{1/3}\alpha_1\int_0^{J(t/T)/J(1)}q(s)^{2/3}\,\dd s\right).
 \end{multline}
 
The following estimates on $G(x,y;t)$ are used in the main text. Both are corollaries of Proposition~\ref{prop:fund_estimate}.
Recall the constant $\kappa$ from \eqref{eq:s_0}.

\begin{corollary}
 \label{cor:G_estimate_1}
 For large $T$, we have for all 
%$x,y\in[0,T^{1/3}]$, 
$x,y\geq 0$ and 
$t\in[\kappa T^{2/3},T]$, 
%$t\in[J^{-1}(4T^{-1/3}J(1))T,T]$,
\begin{equation}
\label{eq:ihp1} 
 C_0^{-1} T^{-1} xy \Ind_{x,y\leq T^{1/3}} \le G(x,y;t) \le C_0 T^{-1} xy,
\end{equation} 
 where $C_0>0$ depends continuously on $\sigma(0)$, $\sigma(1)^{-1}$, $(\inf_{t\in[0,1]}q(t))^{-1}$, $\sup_{t\in[0,1]}|q'(t)|$ and $q(0)$. 
In particular, with the same assumptions,
\begin{equation}
\label{eq:ihp2} 
  %C_1 T^{-1} x \le \frac \dd {\dd y} G(x,y;t)\Big|_{y=0} \le C_2 T^{-1} x.
   \frac \dd {\dd y} G(x,y;t)\Big|_{y=0} \le C_0 T^{-1} x.
\end{equation} 
\end{corollary}

\begin{corollary}
 \label{cor:G_estimate_2}
 For large $T$, we have for all 
$x\geq 0$,
%$x\in[0,T^{1/3}]$, 
$t\in[\kappa T^{2/3},T]$,
\begin{equation}
\label{eq:ihp3} 
  \int_0^\infty G(x,y;t) y\,\dd y \le C_0 T^{-1} x,
\end{equation} 
 where $C_0$ is as in the previous corollary.
\end{corollary}

\begin{proof}[Proof of Corollary~\ref{cor:G_estimate_1}]
Throughout the proof, we will use the fact that $|\psi_n^q(x)| \le \sqrt q x$ for every $q\ge0$, $x\ge 0$ and $n\in\N^*$, by the third part of Lemma~\ref{lem-bullets}. Note that for $t\ge \kappa T^{2/3}$, we have with $\ep^{-1} = J(1)T^{1/3}$,
\[
 J(t/T)/J(1) = J(1)^{-1}\int_0^{t/T}\tfrac 1 2 \sigma^2(s)\,\dd s \ge J(1)^{-1} \tfrac 1 2 \sigma^2(1) t/T \ge \tfrac 1 2 \sigma^2(1) \kappa \ep = 4\ep,
\]
where the first inequality follows from the fact that $\sigma^2$ is a decreasing function and the last equality follows from the definition of $\kappa$ in \eqref{eq:s_0}.
By \eqref{eq:green} and Proposition~\ref{prop:fund_estimate}, with the notation introduced there, we then have for every $x,y\ge0$,
\begin{align}
\nonumber
 G(x,y;t) &\lesssim T^{-1/3} \psi_1^{q(0)}(T^{-1/3}x)\left(\psi_1^{q_{*}(t)}(T^{-1/3}y) + \ep \sum_{n=2}^\infty c_{*n}|\psi_n^{q_{*}(t)}(T^{-1/3}y)|\right)\\
\label{eq:genevieve}
 &\le C_0 T^{-1} xy\sqrt {q(0)} \sup_{t\in[0,1]}\sqrt {q(t)}.
\end{align}
Here, $C_0$ is a constant as in the statement of the corollary; this follows from the fact that the quantities $Q_1^{-1}$ and $Q_2$ in Proposition~\ref{prop:fund_estimate} can be expressed as
\[
 Q_1^{-1} = \left(\inf_{t\in[0,1]}q(t)\right)^{-2/3},\quad Q_2 = \sup_{t\in[0,1]}|q'(t)/q(t)| \le \frac{\sup_{t\in[0,1]}|q'(t)|}{\inf_{t\in[0,1]}q(t)}.
\]
Together with the fact that $\sup_{t\in[0,1]} q(t) \le q(0) + \sup_{t\in[0,1]}|q'(t)|$, Equation \eqref{eq:genevieve} now implies the right-hand inequality of \eqref{eq:ihp1}.

As for the left-hand inequality in \eqref{eq:ihp2}, we have by Proposition~\ref{prop:fund_estimate}, again with the notation introduced there, for every $x,y\ge0$,
\begin{align}
\label{eq:blase}
 G(x,y;t) &\gtrsim T^{-1/3} \psi_1^{q(0)}(T^{-1/3}x)\left(\psi_1^{q^*(t)}(T^{-1/3}y) -\varepsilon \sum_{n=2}^\infty c_n^*|\psi_n^{q^*(t)}(T^{-1/3}y)|\right).
\end{align}
Now note that the function $\psi_1$ is by definition (strictly) positive on $(0,\infty)$ and continuous on $[0,\infty)$. Furthermore, $\psi_1'(0) = 1$ by the first part of Lemma~\ref{lem-bullets}. This implies that the function $x\mapsto \psi_1(x)/x$ can be extended to a continuous and strictly positive function on $[0,\infty)$. In particular,  for every $q>0$, $\psi_1(x) \ge C^{-1} x$ for all $x\in[0,q]$, where $C = \inf_{x\in[0,q]}\psi_1(x)/x>0$ depends continuously on $q$.

Letting $\ep \to 0$ (i.e. $T\to\infty$), the left-hand inequality of \eqref{eq:ihp1} now readily follows from \eqref{eq:blase} by a reasoning similar to the one used above for the right-hand inequality of \eqref{eq:ihp1}, taking into the account the above lower bound on $\psi_1$.

Equation \eqref{eq:ihp2} immediately follows from \eqref{eq:ihp1}.
\end{proof}

\begin{proof}[Proof of Corollary~\ref{cor:G_estimate_2}]
 Similar to the proof of the last corollary, using in addition the fourth part of Lemma~\ref{lem-bullets}. We omit the details.
\end{proof}

 %\begin{remark}
%We can get rid of the $O(1)$ term by taking the first and last interval of size $\delta\ep$ instead of $\ep$ and letting %first $\ep\to0$, then $\delta\to0$. Is it worth it? Let's see....
%\end{remark}

%\begin{lemma}
%Same assumptions as in Proposition~\ref{prop:fund_estimate}. Then, 
%\[
%\int_0^{4\ep} \ep^{-1}\partial_yG(x,0;t)\,\dd t\le C \Ai(Q_1^{1/3}x).
%\]
%\end{lemma}
%\begin{proof}
%Let \(G_{1}(x,y)\) be the Green function of the elliptic operator $L_{Q_1}$ defined above with Dirichlet boundary conditions %at $0$ and $\infty$.
%By the parabolic maximum principle, we have
%\[
%\int_0^{4\ep} \ep^{-1}\partial_yG(x,0;t)\,\dd t \le \int_0^4 \partial_yG(x,0;\ep t)\,\dd t\le  \partial_y G_{1}(x,0) =  %3^{2/3}\Gamma(2/3)\Ai(Q_1^{1/3}x).
%\]
%This yields the proof of the lemma.
%\end{proof}

%\begin{remark}
%We now want to use as initial data $\delta(x-a\ep)$. This is immediately possible. However, we also want to  estimate $\int_0^{b\ep} %u(1,x)\,\dd x$, and for this just having the $L^2$-norm of the higher-order terms ($n\ge 2$) is not enough. We need to have %summability of the coefficients $c_n$, not only of their squares. For this, either 1) use Sobolev norms as Ofer indicated, %or 2) treat the last piece through approximation by a constant and using maximum principles, which yields stretched exponential %decrease of the coefficients.
%\end{remark}

\section{Convergence of the derivative Gibbs measure of (time--homogeneous) branching Brownian motion}
\label{sec:derivative_Gibbs}

In this section, we consider branching Brownian motion with (time-homogeneous) variance $\sigma^2 = 1$, drift $+1$ and reproduction law and branching rate as before. In particular, the left-most particle drifts off to $+\infty$ with zero speed, i.e.\ if $M_t = \min_{u\in\mathcal N(t)}X_u(t)$, then almost surely, as $t\to\infty$, $M_t/t\to 0$ and $M_t\to+\infty$ \cite{Bramson78}. Define the \emph{derivative Gibbs measure} at time $t$:
\[
\mu_t = \sum_{u\in\mathcal N(t)} X_u(t) e^{-X_u(t)}\delta_{X_u(t)/\sqrt t}
\]
The quantity $D_t = \int 1 \dd \mu_t$ is then known as the \emph{derivative martingale}, and it is known \cite{LS87,Neveu,YR11} that $D_t$ converges almost surely as $t\to\infty$ to a (strictly) positive limit $D_\infty$ whose Laplace transform is given by $\E[\exp(-e^{-x}D_\infty)] = \phi(x)$, where $\phi$ is a solution to \eqref{eq:fkpp}.

Let $\rho$ denote the law of a BES(3) process at time 1, started at 0, i.e.\ 
\[
\rho(\dd x) = \sqrt{\frac 2 \pi}\,x^2 e^{-x^2/2}\Ind_{x\ge 0}\,\dd x.
\]
\begin{theorem}
\label{th:Gibbs}
In probability, $\mu_t$ converges weakly to $D_\infty \rho.$ Moreover, for every family $(f_t)_{t\ge 0}$ of uniformly bounded measurable functions (i.e.\ $\sup_{t,x}|f_t(x)| < \infty$), we have 
\[
\int f_t \,\dd\mu_t - D_\infty \int f_t \,\dd \rho \to 0,\quad\text{in probability.}
\]
\end{theorem}

\begin{remark}
Convergence in probability of the Gibbs measure 
\[
\mu^*_t = \sqrt t \times \sum_{u\in\mathcal N(t)} e^{-X_u(t)}\delta_{X_u(t)/\sqrt t},
\]
has recently been shown by Madaule \cite{Madaule} for general branching random walks. While Theorem~\ref{th:Gibbs} (at least the first statement) could be in principle recovered from the results in \cite{Madaule} (see in particular Proposition~3.4 of that paper), we present below for completeness a fairly simple proof.
\end{remark}

\begin{proof}
Note that we can (and will) assume w.l.o.g.\ that $f_t \ge 0$ for each $t\ge0$.
For every $s \le t$, define the measure
\[
\mu_t^s = \sum_{u\in\mathcal N(t)} X_u(t) e^{-X_u(t)}\Ind_{(X_u(r)\ge 0\ \forall s\le r\le t)}\delta_{X_u(t)/\sqrt t}
\]
Since $\min_{u\in\mathcal N(t)}X_u(t)\to +\infty$ almost surely \cite{M75}, there exists a random time $S$, such that we have $\mu_t^s = \mu_t$ for all $S\le s\le t$. 
Since moreover $D_s\to D_\infty$ almost surely, as $s\to\infty$, 
it is enough to show that almost surely, for any family of nonnegative
functions $(f_t)_{t\ge0}$ as in the statement of the theorem,
\begin{equation}
\label{eq:30987}
\lim_{s\to\infty}\limsup_{t\to\infty} 
\left|\E[e^{-\int f_t \,\dd\mu^s_t}\,|\,\F_s] - 
e^{-D_s \int f_t \,\dd \rho}\right| = 0,
\quad a.s.\end{equation}

Let $s\le t$. Define $f_{s,t}(x) = f_t(x\sqrt{(t-s)/t})$. 
By the branching property and  Jensen's inequality,
\begin{equation}
\label{eq:1309876}
\E[e^{-\int f_t\dd\mu_{t}^s}\,|\,\F_s] = \prod_{u\in\mathcal N(s)} \E_{X_u(s)}[e^{-\int f_{s,t}\,\dd \mu^0_{t-s}}] \ge \exp\left(-\sum_{u\in\mathcal N(s)} \E_{X_u(s)}\left[\int f_{s,t}\dd\mu^0_{t-s}\right]\right),
\end{equation}
We now have for every $x\ge 0$, by the first moment formula for branching Markov processes \cite[Theorem~4.1]{INW3} and Girsanov's theorem, for every bounded measurable function $f$,
\begin{equation*}
%\label{eq:234}
\begin{split}
\E_x\left[\int f\dd\mu^0_t\right] &= e^{t/2}E_x[(B_t+t)e^{-(B_t+t)}f((B_t+t)/\sqrt t)\Ind_{(B_r\ge 0\ \forall r\le t)}]\\
 &= e^{-x}E_x[B_tf(B_t/\sqrt t)\Ind_{(B_r\ge 0\ \forall r\le t)}]\\
 &=xe^{-x} E_x[f(R_t/\sqrt t)] = xe^{-x} E_{x/\sqrt t}[f(R_1)],
\end{split}
\end{equation*}
where under $P_x$, $(R_t)_{t\ge 0}$ is a three-dimensional Bessel process started at $x$ \cite[Section~XI.1]{RevuzYor}. The law of $R_1$ under $P_x$ has a continuous density with respect to Lebesgue measure for every $x$ which converges uniformly to the density of $\rho$ as $x\to0$. It follows easily from this that for every $x\ge 0$,
\begin{equation}
\label{eq:5987}
\E_x\left[\int f_{s,t}\,\dd \mu^0_{t-s}\right] - xe^{-x} \int f_{t}\,\dd \rho \to 0,\quad\tas t\to\infty.
\end{equation}
 Equations \eqref{eq:1309876} and \eqref{eq:5987} 
 now yield the inequality ``$\ge$'' in \eqref{eq:30987}. 
 In order to obtain the other inequality, we have by Lemma~\ref{lem:478} below, for some constants $C,C'$,
\begin{equation}
\label{eq:8884}
\E_x\left[\left(\int f_{s,t}\,\dd \mu^0_{t-s}\right)^2\right] \le C'\E^0_x[D_{t-s}^2] \le Ce^{-x},
\end{equation}
where the superscript in $\E^0_x$ indicates that the particles are killed upon hitting the origin.
By the branching property and the inequalities $e^{-x} \le 1-x+x^2 \le e^{-x+x^2}$ for $x\ge 0$, we then get by \eqref{eq:8884},
\begin{align}
  \label{eq-exp22}
  \nonumber
\E[e^{-\int f_t\dd\mu_{t}^s}\,|\,\F_s] &\le \exp\left(
\sum_{u\in\mathcal N(s)} -\E_{X_u(s)}\left[\int f_{s,t}\dd\mu^0_{t-s}\right] 
+ \E_{X_u(s)}\left[\left(\int f_{s,t}\,\dd \mu^0_{t-s}\right)^2\right]\right)\\
&\le \exp\left(
-D_s \int f_{t}\,\dd \rho 
+ C W_s + E_{s,t}\right),
\end{align}
where $W_s = \sum_{u\in\mathcal N(s)} e^{-X_u(s)}$ and 
$$E_{s,t}=
\sum_{u\in\mathcal N(s)} -\E_{X_u(s)}\left[\int f_{s,t}\dd\mu^0_{t-s}\right] +
D_s \int f_{t}\,\dd \rho$$ 
is an $\F_s$-measurable term.
By \eqref{eq:8884} and the fact that $\rho$ has a continuous density with 
respect to Lebesgue's measure,
$E_{s,t}$ tends to zero almost surely, as $t\to\infty$, for each fixed $s$. 
Since $W_s\to0$ almost surely, as $s\to \infty$ 
(see e.g.\ \cite{LS87,Neveu}), 
the inequality \eqref{eq-exp22}
yields the inequality ``$\le$'' in \eqref{eq:30987}. 
This finishes the proof of the theorem.
%\begin{equation}
%\label{eq:484}
%\liminf_{t\to\infty}\E[e^{-\int f\dd\mu_{t}^s}\,|\,\F_s] = \exp(-D_s E_0[f(R_1)]).
%\end{equation}
%By \eqref{eq:484} and dominated convergence,
%\[
%\liminf_{t\to\infty} \E[e^{-D_t^{-1}\int f\,\dd\mu_t}] \ge \liminf_{s\to\infty}\liminf_{t\to\infty} \E[e^{-D_s^{-1}\int f\,\dd\mu_t^s}] \ge \exp(-E_0[f(R_1)]).
%\]
%But since $D_t^{-1}\mu_t$ is a probability measure, this implies that $D_t^{-1}\mu_t$ converges weakly to the law of $R_1$. The almost sure convergence of $D_t$ to $D_\infty$ now implies the theorem.
\end{proof}

\begin{lemma}
\label{lem:478} Let $\E^0_x$ be the law of BBM as in the beginning of this section but where in addition particles are killed upon hitting the origin. For some constant $C$, $\E^0_x[D_{t}^2] \le Ce^{-x}$ for every $x\ge 0$ and $t\ge0$.
\end{lemma}
\begin{proof}
We first note that $(D_t)_{t\ge 0}$ is a martingale as well under $\E^0_x$. In particular, $\E^0_x[D_t] = xe^{-x}$ for every $x\ge0$ and $t\ge 0$. By the second moment formula for branching Markov processes \cite[Theorem~4.15]{INW3}, this gives for some constant $C$,
\begin{align*}
\E^0_x[D_{t}^2] = \E^0_x\Big[\sum_{u\in\mathcal N(t)} X_u(t)^2 e^{-2X_u(t)}\Big] + C \E^0_x\Big[\int_0^t\sum_{u\in\mathcal N(s)} X_u(s)^2 e^{-2X_u(s)}\,\dd s\Big].
\end{align*}
By the first moment formula for branching Markov processes and Girsanov's theorem we get as in the proof of Theorem~\ref{th:Gibbs},
\begin{align*}
%\label{eq:8749}
\E^0_x[D_{t}^2] = e^{-x}\left(E_x[B_t^2e^{-B_t}\Ind_{B_s \ge 0\,\forall s\le t}] + C E_x\Big[\int_0^{t\wedge T_0} B_s^2 e^{-B_s}\,\dd s \Big]\right),
\end{align*}
where $T_0$ is the first hitting time of the origin. The term in the first expectation is bounded by a constant. As for the second expectation, by the inequality $x^2e^{-x} \le C'e^{-x/2}$ and Ito's formula, we have
\[
E_x\Big[\int_0^{t\wedge T_0} B_s^2 e^{-B_s}\,\dd s \Big] \le 4C'E_x[e^{-B_{t\wedge T_0}/2} - e^{-x/2}] \le 4C'.
\]
This yields the lemma.
%As for the second expectation, denote by $(R_t)_{t\ge 0}$ a three-dimensional Bessel process. By Williams' result on time-reversal of Brownian motion \cite[Corollary VII.4.6]{RevuzYor},
%\[
%E_x\Big[\int_0^{t\wedge T_0} B_s^2 e^{-B_s}\,\dd s \Big] \le E_x\Big[\int_0^{T_0} B_s^2 e^{-B_s}\,\dd s \Big] \le E_0\Big[\int_0^\infty R_s^2 e^{-R_s}\,\dd s \Big],
%\]
%and the last quantity is easily seen to be finite by calculation.
\end{proof}

\end{document}